\documentclass{article}
\usepackage{blindtext}

\usepackage{amsmath, amssymb, amscd, amsthm, graphicx, amsfonts, xy, xcolor}

\newtheorem{theorem}{Theorem}
\newtheorem{lemma}[theorem]{Lemma}
\newtheorem{corollary}[theorem]{Corollary}

\newtheorem{remark}[theorem]{Remark}

\newcommand{\RNumb}[1]{\uppercase\expandafter{\romannumeral #1\relax}}

\title{Tiling of regular polygon with similar right triangles \RNumb{1}}
\date{}
\author{Vasenov Ivan\thanks {This work was prepared in frame of math circle "Math and Olympiads" https://mccme.ru/circles/oim/. High School 67, Kutuzovsky prospect, 10, Moscow, Russia. }}
		
\begin{document}

\maketitle

{\it A tiling} is a decomposition of a polygon into finitely many non-overlapping triangles. (Note that a tiling is not a triangulation, i.e. a vertex of a triangle can lie on a side of a triangle.)  

Clearly, for each $n$ there is the `trivial' tiling of a regular $n$-gon obtained by  
joining the center of the  $n$-gon to each vertex of the $n$-gon. 
We obtain $n$ congruent isosceles triangles with angles $\frac{2\pi}{n}$ at their  vertices opposite to their bases.  The `trivial' tiling with isosceles triangles gives a `trivial' tiling  with congruent right triangles having angles $\frac{\pi}{n}$.

\begin{theorem}[M.Laczkovich]\label{t:lazkovich}
If a regular $n$-gon, $n \geq 25, n \neq 30,42$, can be tiled with similar right triangles, then the smaller angle of these triangles equals to  $\frac{\pi}{n}$.
\end{theorem}

Theorem \ref{t:lazkovich} follows from [L20, Theorem 1], since for any tiling of the regular $n$-gon with similar right triangles, the angles of the triangles are rational multiples of $\pi$. (Let us present a simple proof of this fact. By $p,q,r$ we denote the number of smaller acute angles $\alpha$, bigger acute angles $\frac{\pi}{2}-\alpha$ and angles $\frac{\pi}{2}$ at the vertex of the regular $n$-gon, respectively. Take $a = 2\alpha/\pi$, then by Lemma \ref{t:lemma5} it follows that $\alpha$ is a rational multiple of $\pi$).

\begin{theorem}\label{t:tiling}
If a regular n-gon, $n \geq 5$, $n \neq 28$, can be tiled with similar right triangles, then the smaller angle of these triangles is in $\left\{\frac{\pi}{n},\frac{2\pi}{n}, \frac{\pi} {6}+\frac{2\pi}{3n}\right\}$.
\end{theorem}

Because of Theorem \ref{t:lazkovich}, Theorem \ref{t:tiling} is a new result only for  $5 \leq n \leq 24$ or $n \in \{30, 42\}$ (see Remark \ref{t:remark} (e)). Theorem \ref{t:tiling} was announced in [V19].


Theorem \ref{t:tiling} and [V20, Theorem 1] imply the following corollary.

\begin{corollary}\label{t:color}
	If a regular $n$-gon, $n \geq 9 $, $n \neq 12, 14, 20, 32, 44$, can be tiled with similar right triangles, then the smaller angle of these triangles is in $\{\frac{\pi}{n}, \frac{2\pi}{n}\}$.
\end{corollary}

 We do not know if the remaining value $\frac{2\pi}{n}$ from Corollary \ref{t:color} is `realizable' by some tilings for  $9 \leq n \leq 24$ or $n = 30,42$.

 Our proof of Theorem \ref{t:tiling} is based on showing that for angles of the triangles other than mentioned in statement \textit{the number of smaller acute angles} is  greater than the \textit{number of larger acute angles}. This idea reduces the amount of cases to search through.

We illustrate the idea by a proof of the following corollary. 

\begin{corollary}\label{t:collary}
	If a regular 8-gon can be tiled with similar right triangles, then the smaller angle of these triangles is in $\left\{\frac{\pi}{8},\frac{\pi}{4}\right\}$.
\end{corollary}

If a reader is only interested in the proof of Theorem \ref{t:tiling}, he could skip Lemma \ref{t:lemma3} and the proof of Corollary \ref{t:collary}.

By the side of a triangle we mean side of a triangle without its vertex.

\begin{lemma}\label{t:lemma3}
If $\frac{3}{2}=pa+q\left(1-a\right)+r$, where $0<a<\frac{1}{2}$, $a\neq\frac{1}{4}$ and $p,q,r$ are non-negative integers, then $p>q$ and $a=\frac{3}{2s}$ for some positive integer $s$.
\end{lemma}

\begin{proof}
Assume to the contrary that $q-p \geq 0$. 
We have 
$$\frac{3}{2}=pa+q\left(1-a\right)+r=\left(q-p\right)\left(1-a\right)+p+r.$$ 
If $q-p=0$ then $\frac{3}{2}=p+r$, but $p+r$ is an integer.
Then $q-p\ge1$.

Since $\frac12 < 1-a=\frac{\frac{3}{2}-p-r}{q-p}\le\frac32-p-r$, we have $r=p=0$.  

Thus $\frac{3}{2}=(1-a)q$. $q \neq 0$ since $\frac{3}{2} \neq 0$.
Since $0<a<\frac{1}{2}$, it follows that $\frac{q}{2}<\frac32<q$. Then $q=2$. 
Hence $a=\frac14$.  
A contradiction.
 		
Thus  $p>q$. We have
$$\frac{3}{2}=pa+q\left(1-a\right)+r=\left(p-q\right)a+q+r\ .$$

Since $0<a$  and $q,r$ are non-negative integers, it follows that $0 \leq q+r <2$. 
Then $q+r\in\{0,1\}$.
Hence $a=\frac{3-2(q+r)}{2(p-q)}\in\{\frac3{2(p-q)},\frac1{2(p-q)}\}$. 
 
Thus $a=\frac{3}{2s}$, where $s$ is a positive integer.	
\end{proof}

\begin{lemma}\label{t:lemma4}
If $4=pa+q\left(1-a\right)+r$, where $0<a<\frac{1}{2}$, 
$a \not\in\{\frac{1}{4},\frac{1}{5},\frac{2}{5},\frac{3}{7},\frac{1}{3}\}$,  $p,q,r$ are non-negative integers,  then $p\geq q$.
\end{lemma}

\begin{proof}
Assume to the contrary that $q-p > 0$. 
Then $q-p\geq 1$.
We have 
$$4=pa+q\left(1-a\right)+r=\left(q-p\right)\left(1-a\right)+r+p.$$

If $r+p>2$, then  $1-a=\frac{4-p-r}{q-p}$ is either greater than or equal to 1 or is less than or equal to $\frac12$.  

If $r+p=2$, then $1-a=\frac{2}{q-p}$. Since $0<a<\frac12$ it follows that $a=\frac13$.

If $r+p=1$, then $1-a=\frac{3}{q-p}$. Hence $a\in\{\frac14,\frac25\}$.

If $r+p=0$, then $1-a=\frac{4}{q-p}$. Hence $a\in\{\frac15, \frac13, \frac37\}$.

A contradiction. Thus, $p \geq q$.
\end{proof}

\begin{proof}[Proof of Corollary \ref{t:collary} independent of Theorem \ref{t:tiling}] \label{t:proofcol}
Suppose there is a tilling of a regular $8$-gon with right triangles of angles $\alpha < \frac\pi4$, $\frac{\pi}{2}-\alpha$ and  $\frac{\pi}{2}$, and  $\alpha \neq \frac\pi8$. 

Take a vertex of the 8-gon. By $p,q,r$ we denote the number of smaller acute angles $\alpha$, bigger acute angles $\frac{\pi}{2}-\alpha$ and angles $\frac{\pi}{2}$ at this vertex, respectively. Then $\frac{3\pi}{4}=p\alpha+q\left(\frac{\pi}{2}-\alpha\right)+r\frac{\pi}{2}$. Divide this equality by $\frac{\pi}{2}$ and get $\frac{3}{2}=pa+q\left(1-a\right)+r$, where $0<a<\frac{1}{2}$.

Then by Lemma \ref{t:lemma3} 

$\bullet$ at each vertex of an 8-gon the number of smaller acute angles $\alpha$ is greater than the number of bigger acute angles $\frac{\pi}{2}-\alpha$.

$\bullet$ $\alpha=\frac{3\pi}{4s}$ for some positive integer $s$. 

Hence $\alpha\not\in\{\frac{\pi}{10},\frac{\pi}{5},\frac{3\pi}{14},\frac{\pi}{6}\}$.

For the triangles that have same vertices inside of the $8$-gon and not on the side of a triangle we use Lemma \ref{t:lemma4}. Then at any point inside of the $8$-gon and not on the side of a triangle the number of smaller acute angles is greater than or equal to the number of bigger acute angles. 
   
For the triangles that have same vertices on the side of $8$-gon or on the side of a triangle, we use Lemma \ref{t:lemma4}, substituting $r-2$ for $r$. Then at any point on the side of $8$-gon or on the side of a triangle the number of smaller acute angles is greater or equal to the number of bigger acute angles.

Hence the number of smaller acute angles is greater than the number of bigger acute angles. A contradiction. 
\end{proof}

\begin{lemma}\label{t:lemma5}
For any integer $n\geq5$ if $2-\frac{4}{n}=pa+q(1-a)+r$, where $0<a\leq \frac12$, $a\not\in\{\frac{2}{n},\frac{4}{n}, \frac{1} {3}+\frac{4}{3n}\}$, $p,q,r$ are non-negative integers, then $p>q$ and $a\in\{ \frac{2-\frac{4}{n}}{s},\frac{1-\frac{4}{n}}{s}\}$ for some positive integer $s$.
\end{lemma}

\begin{lemma}\label{t:lemma6}
	For any integer $n\geq5, n\neq28$ if  $a\in\{\frac{1}{4},\frac{1}{5},\frac{2}{5},\frac{3}{7},\frac{1}{3}\}$ and $a$ $\in\{ \frac{2-\frac{4}{n}}{s},\frac{1-\frac{4}{n}}{s}\}$ for some positive integer $s$, then either $a$ or $1-a$ is in $\{\frac{2}{n},\frac{4}{n}, \frac{1} {3}+\frac{4}{3n}\}$.
\end{lemma}

Lemma \ref{t:lemma5} and Lemma \ref{t:lemma6} will be proved after the proof of Theorem \ref{t:tiling}.

\begin{proof}[Proof of Theorem \ref{t:tiling}]
	
	Suppose  to the contrary there is a tilling of a regular $n$-gon with right triangles of angles $\alpha \leq \frac\pi4$, $\frac{\pi}{2}-\alpha$ and  $\frac{\pi}{2}$, and $\alpha\not \in\left\{\frac{\pi}{n},\frac{2\pi}{n}, \frac{\pi} {6}+\frac{2\pi}{3n}\right\}$. Denote $a=\frac{2\alpha}{\pi}$.
	
	If $\alpha = \frac{\pi}{4}$, then the angle of regular $n$-gon should be a rational multiple of $\frac{\pi}{4}$, hence $n \in \{4, 8\}$. $n \geq 5$, for $n = 8$, $\alpha = \frac{2\pi}{8} = \frac{\pi}{4}$.
	
	Take a vertex of the $n$-gon. By $p,q,r$ we denote the number of smaller acute angles $\alpha$, bigger acute angles $\frac{\pi}{2}-\alpha$ and angles $\frac{\pi}{2}$, respectively at this vertex. Then $2-\frac{4}{n}=pa+q(1-a)+r$, where $0<a<\frac12$.  Then by Lemma \ref{t:lemma5}, $p>q$  and $\alpha\in\{ \frac{\pi-\frac{2\pi}{n}}{s},\frac{\frac{\pi}{2}-\frac{2\pi}{n}}{s}\}$ for some positive integer $s$.

Since $\alpha\not \in\{\frac{\pi}{n},\frac{2\pi}{n}, \frac{\pi} {6}+\frac{2\pi}{3n}\}$, by Lemma \ref{t:lemma6} it follows that, $\alpha\not\in\{\frac{\pi}{10},\frac{\pi}{5},\frac{3\pi}{14},\frac{\pi}{6},\frac{\pi}{8}\}$.

 For the triangles that have same vertices inside of regular $n$-gon not on the side of a triangle, denote $a=\frac{2\alpha}{\pi}$, by $p$ we denote the number of the smaller acute angles $\alpha$, by $q$ the number of greater acute angles $\frac{\pi}{2}-\alpha$, by $r$ the number of right angles. Then by Lemma \ref{t:lemma4} at any point inside of $n$-gon not on the side of a triangle $p\geq q$. 
 
 For the triangles that have same vertices on the side of $n$-gon or on the side of a triangle, we use Lemma \ref{t:lemma4}, substituting $r-2$ for $r$. Then at any point on the side of $n$-gon or on the side of a triangle $p\geq q$. 
 
 Hence the number of smaller acute angles is greater than the number of bigger acute angles. A contradiction.

\end{proof}

\begin{proof}[Proof of Lemma \ref{t:lemma5}]
Assume to the contrary that  $q-p\geq0$.

We have 
$$2-\frac{4}{n}=pa+q\left(1-a\right)+r=\left(q-p\right)\left(1-a\right)+p+r.$$ 

If $q-p=0$, then $2-\frac{4}{n}=p+r$, but $2>2-\frac{4}{n}>1$ and $p+r$ is integer. Then $q-p\geq 1$.

If $r+p>1$, then $1-a=\frac{2-\frac{4}{n}-p-r}{q-p}<0$. A contradiction. Thus either $r+p=1$ or $r+p=0$.

If $r+p=1$, then $1-a=\frac{1-\frac{4}{n}}{q-p}$. Hence $a=\frac{4}{n}$.

If $r+p=0$, then $1-a=\frac{2-\frac{4}{n}}{q-p}$. Hence $q-p\in\{2,3\}$. Then $a\in\{\frac{2}{n}, \frac13 +\frac{4}{3n}\}$.

A contradiction. Thus $p>q$.

We have
$$2-\frac{4}{n}=pa+q\left(1-a\right)+r=\left(p-q\right)a+q+r=sa+u\ .$$

Where $u=q+r$ and $s=p-q$. Hence $0 \leq r+q<2$.

Then $a=\frac{2-u-\frac4n}{s}$.

Since $2-u \in \{1, 2\}$, lemma is proved.

\end{proof}

\begin{proof}[Proof of Lemma \ref{t:lemma6}]

We provide proof for $z = as + \frac{4}{n} \in \{1,2\}$ so that $a=\frac{z-\frac{4}{n}}{s}$.

Suppose that $a=\frac{1}{4}=\frac{z-\frac{4}{n}}{s}$. Hence $s=4z-\frac{16}{n}$. Since $s$ and $n$ are positive integers, it follows that $n\in\{1,2,4,8,16\}$. But $n \geq 5$ and $a=\frac{1}{4}=\frac{2}{8}=\frac{4}{16}$. Thus for $a=\frac{1}{4}$, $a\in\{\frac{2}{n},\frac{4}{n}, \frac{1} {3}+\frac{4}{3n}\}$ for some positive integer $n$. 

Suppose that $a=\frac{1}{5}=\frac{z-\frac{4}{n}}{s}$. Hence $s=5z-\frac{20}{n}$. Since $s$ and $n$ are positive integers, it follows that $n\in\{1,2,4,5,10,20\}$. But $n \geq 5$ and $a=\frac{1}{5}=1-\frac45=\frac{2}{10}=\frac{4}{20}$. Thus for $a=\frac{1}{5}$, $a\in\{\frac{2}{n},\frac{4}{n}, \frac{1} {3}+\frac{4}{3n}\}$ for some positive integer $n$. 

Suppose that $a=\frac{2}{5}=\frac{z-\frac{4}{n}}{s}$. Hence $2s=5z-\frac{20}{n}$. Since $s$ and $n$ are positive integers, it follows that $n\in\{1,2,4,5,10,20\}$. But $n \geq 5$ and $a=\frac{2}{5}=\frac{2}{5}=\frac{4}{10}=\frac{1}{3}+\frac{4}{3\cdot20}$. Thus for $a=\frac{2}{5}$, $a\in\{\frac{2}{n},\frac{4}{n}, \frac{1} {3}+\frac{4}{3n}\}$ for some positive integer $n$. 

Suppose that $a=\frac{3}{7}=\frac{z-\frac{4}{n}}{s}$. Hence $3s=7z-\frac{28}{n}$. Since $s$ and $n$ are positive integers, it follows that $n\in\{1,2,4,7,14,28\}$. But $n \geq 5$, $n \neq 28$ and $a=\frac{3}{7}=\frac{1} {3}+\frac{4}{3\cdot14}=1-\frac{4}{7}$. Thus for, $a=\frac{3}{7}$ $a\in\{\frac{2}{n},\frac{4}{n}, \frac{1} {3}+\frac{4}{3n}\}$ for some positive integer $n$. 

Suppose that $a=\frac{1}{3}=\frac{z-\frac{4}{n}}{s}$. Hence $s=3z-\frac{12}{n}$. Since $s$ and $n$ are positive integers, it follows that $n\in\{1,2,3,4,6,12\}$. But $n \geq 5$ and $a=\frac{1}{3}=\frac{2}{6}=\frac{4}{12}$. Thus for $a=\frac{1}{3}$, $a\in\{\frac{2}{n},\frac{4}{n}, \frac{1} {3}+\frac{4}{3n}\}$ for some positive integer $n$. 

\end{proof}

\begin{remark}(Relation to earlier results)\label{t:remark} 
		
	
	(a) A tiling of a regular $n$-gon with right non-isosceles triangles with angle $\alpha\ne\frac{\pi}{4}$ is called {\it regular} if either 
	
	\quad $(a1)$ at each point of the (convex hull of the) $n$-gon the number of angles $\alpha$ is equal to the number of angles $\frac{\pi}{2}-\alpha$, or
	
	 \quad$(a2)$ at each point of the $n$-gon the number of angles $\alpha$ is equal to the number of angles $\frac{\pi}{2}$, or
	
	 \quad$(a3)$ at each point of the $n$-gon the number of angles $\frac{\pi}{2}-\alpha$ is equal to the number of angles $\frac{\pi}{2}$.

	(b)	From [L12, Theorem 2.1] it follows that a regular $n$-gon, $n \geq 5, n \neq 6$ cannot be regularly tiled with congruent right triangles.
	
	(c) Let us present a simple deduction of (b) from [L12, Theorem 2.1]. We consider cases (i,ii, ..., ix) listed in [L12, Theorem 2.1].
	If a regular $n$-gon, $n \geq 5, n\neq 6$, is regularly tiled with triangles with angles $\alpha, \frac{\pi}{2}-\alpha, \frac{\pi}{2}$, then one of the properties (a1), (a2), (a3) holds. Since $n \geq 5, n\neq 6$ the cases (i, iv, vi, ix) do not hold. Since the $n$-gon is regular and triangles are right, the cases (vii, viii) do not hold.
	
	In the case (ii) we have either $\alpha=\pi-\frac{2\pi}{n}$ or $\frac{\pi}{2}-\alpha=\pi-\frac{2\pi}{n}$ or $\frac{\pi}{2}=\pi-\frac{2\pi}{n}$. Since $n \geq 5$ it follows that $\pi-\frac{2\pi}{n} > \frac{\pi}{2}$. A  contradiction.
	
	In the case (iii) we have either $\alpha=\frac{2\pi}{n}$ or $\frac{\pi}{2}-\alpha=\frac{2\pi}{n}$; in both cases $\sin\alpha$, $\sin(\frac{\pi}{2}-\alpha)$, $\sin\frac{\pi}{2}$ should be commensurable. Since $\sin\frac{\pi}{2} = 1$ it follows that $\sin(\frac{\pi}{2}-\alpha)=\cos\alpha$ should be rational. Since $n\not\in\left\{1,2,3,4,6\right\}$, it follows that $\cos\frac{2\pi}{n}$ is not rational, see for example [S21, Theorem 4.1]. Hence the case (iii) does not hold.
	
	In the case (v) we have either $\alpha=\frac{2\pi}{n}$ or $\frac{\pi}{2}-\alpha=\frac{2\pi}{n}$ and $\frac{\sin\alpha}{\sin(\frac{\pi}{2}-\alpha)}=1$. Thus $\alpha = \frac{\pi}{2}-\alpha = \frac{\pi}{4}$, but $n>4$. Hence the case (v) does not hold.
	
	(d) In Theorem \ref{t:lazkovich} we have the assumption that $n \geq 25, n \neq 30, 42$. In [L12, Theorem 2.1] we have the assumption that the tiling is regular. We do not know if for any $5 \leq n \leq 24, n \in {30, 42}$ there are irregular tilings with right triangles with angle $\alpha\neq\frac{\pi}{n}$.
	
	(e) For the cases $n<5$ and $n=6$ see [L12, Theorem 2.1] and [S01, Theorem].
	
\end{remark}

[L12] M.Laczkovich, Tilings of Convex Polygons with Congruent Triangles, Discrete and Computational Geometry  48, (2012) 330-372 

[L20] M.Laczkovich, Irregular tilings of regular polygons with similar triangles, to appear in Discrete and Computational Geometry 

[S21] A. Skopenkov, Mathematics Through Problems: from olympiades and math circles to profession. Part I. Algebra. 2021 AMS, Providence,
 \linebreak https://www.mccme.ru/circles/oim/algebra\_eng.pdf

[S01] B. Szegedy, Tilings of the square with similar right triangles, Combinatorica 21 (1) (2001) 139–144

[V19] I. Vasenov, Tiling of regular n-gon with right triangles 
\linebreak https://www.mccme.ru/circles/oim/mmks/works2019/vasenov1.pdf

[V20] L. Vigdorchik, Tiling of regular polygon with similar right triangles \RNumb{2}, 
\linebreak   https://www.mccme.ru/circles/oim/mmks/works2020/vigdorchik1.pdf


\end{document}